\documentclass[reqno,12pt]{amsart}
 
\usepackage{amsmath}
\usepackage{amsthm,amssymb,color,comment}
\usepackage{bbm}
\usepackage{mathrsfs}
\usepackage{hyperref}
\usepackage[TS1,T1]{fontenc}
\usepackage[utf8]{inputenc}
\usepackage{dsfont}
\usepackage{tikz}
\usepackage{enumitem}
\usepackage{subfigure}
\usepackage{mathtools}
\usepackage{bbold}
\usepackage{esint}
\usepackage[sort]{cite}

\numberwithin{equation}{section}

\newtheorem{thm}{Theorem}[section]
\newtheorem{proposition}[thm]{Proposition}
\newtheorem{lem}[thm]{Lemma}

\theoremstyle{definition}
\newtheorem{Ass}[thm]{Assumption}
\newtheorem{rem}[thm]{Remark}

\DeclareMathOperator{\DIV}{div}

\newcommand{\R}{\mathbb{R}}
\newcommand{\Td}{\mathbb{T}^{d}}
\newcommand{\Rd}{\mathbb{R}^{d}}

\newcommand{\p}{\partial}
\newcommand {\f}{\frac}
\newcommand{\eps}{\varepsilon}

\newcommand{\diff}{\mathop{}\!\mathrm{d}}

\newcommand{\weaks}{\overset{\ast}{\rightharpoonup}}

\newcommand{\weak}{\rightharpoonup}

\makeatletter
\newcommand{\doublewidetilde}[1]{{%
  \mathpalette\double@widetilde{#1}%
}}
\newcommand{\double@widetilde}[2]{%
  \sbox\z@{$\m@th#1\widetilde{#2}$}%
  \ht\z@=.9\ht\z@
  \widetilde{\box\z@}%
}
\makeatother
\author{Charles Elbar}
\address{{\it Charles Elbar:} Sorbonne Universit\'{e}, Laboratoire Jacques-Louis Lions (LJLL), F-75005 Paris, France}
\email{charles.elbar@sorbonne-universite.fr}
\thanks{}

\author{Beno\^\i t Perthame}
\address{{\it Beno\^\i t Perthame: } Sorbonne Universit{\'e}, CNRS, Universit\'{e} de Paris, Inria, Laboratoire Jacques-Louis Lions UMR7598, F-75005 Paris}
\email{Benoit.Perthame@sorbonne-universite.fr}

\author{Jakub Skrzeczkowski}
\address{{\it Jakub Skrzeczkowski: } Institute of Mathematics of Polish Academy of Sciences; Faculty of Mathematics, Informatics and Mechanics, University of Warsaw, Poland}
\email{jakub.skrzeczkowski@student.uw.edu.pl}
\thanks{Jakub Skrzeczkowski was supported by National Science Center, Poland through project no. 2019/35/N/ST1/03459.}

\begin{document}

\title[]{On the limit problem arising in the kinetic derivation of the Cahn-Hilliard equation}

\begin{abstract}
The non-local degenerate Cahn-Hilliard equation is derived from the Vlasov equation with long range attraction. We study the local limit as the delocalization parameter converges to $0$. The difficulty arises from the degeneracy which requires compactness estimates, but all necessary a priori estimates can be obtained only on the nonlocal quantities yielding almost no information on the limiting solution itself. We introduce a novel condition on the nonlocal kernel which allows us to exploit the available nonlocal a priori estimates. The condition, satisfied by most of the kernels appearing in the applications, can be of independent interest. Our approach is flexible and systems can be treated as well.
\end{abstract}



\maketitle
\setcounter{tocdepth}{1}

\section{Introduction}
We consider the nonlocal PDE
\begin{equation}\label{eq:nonlocal_PDE_CH}
\partial_{t} \rho_{\eps} - \Delta \rho_{\eps} + \DIV (\rho_{\eps} \nabla \Delta (\rho_{\eps} \ast \omega_{\eps} \ast \omega_{\eps}) ) = 0
\end{equation}
arising in the kinetic derivation of the degenerate Cahn-Hilliard equation via hydrodynamic limit as in our recent paper \cite{elbar-mason-perthame-skrzeczkowski} which  makes the formal approach of Takata and Noguchi \cite{takata2018simple} fully rigorous. Furthermore, PDEs of the form \eqref{eq:nonlocal_PDE_CH} appear in the numerical analysis of the particle method as explained below. Concerning the notation, $\varepsilon>0$ is a small parameter, $\rho_{\varepsilon}: \R^d \times [0,T] \to [0,\infty)$ is the solution while  
$$
\omega_{\varepsilon}(x) = \frac{1}{\varepsilon^d} \omega\left(\frac{x}{\varepsilon}\right)
$$ is a usual mollification kernel (one can think of $\omega$ being smooth, nonnegative and compactly supported function but more general assumption will be presented below). The remaining question is whether one can consider the limit  $\varepsilon \to 0$ in \eqref{eq:nonlocal_PDE_CH} and arrive at
\begin{equation}\label{eq:local_PDE_CH}
\partial_{t} \rho - \Delta \rho + \DIV (\rho \nabla \Delta \rho) = 0.
\end{equation}

\noindent A natural strategy for such problems are so-called commutator estimates, originally studied in the context of renormalized solutions to transport equation \cite{MR1022305} and conservation of energy for weak solutions of Euler equation \cite{MR1298949} or more general hyperbolic systems \cite{MR3814601, MR4038422}. The commutator estimates control the error resulting from the fact that the mollification operator does not commute with nonlinear functions. In our case, they cannot be applied directly, as we do not have sufficient estimates on $\{\rho_{\eps}\}$. Indeed, both energy and entropy identities for \eqref{eq:nonlocal_PDE_CH} yield estimates on non-local quantities $\rho_{\eps} \ast \omega_{\eps}$, $\nabla \rho_{\eps} \ast \omega_{\eps}$ and $\Delta \rho_{\eps} \ast \omega_{\eps}$ (see Section \ref{sect:uniform_estimates}). In particular, we do not even know that $\{\rho_{\eps}\}$ is uniformly bounded in $L^2_{\text{loc}}((0,T)\times \R^d)$ (except in dimension $d=2$). In Appendix~\ref{app:conv_d=2_higher_integrability} we show that for a fairly arbitrary kernel (up to some weak growth assumptions at infinity), the direct commutator estimates work in dimension $d=2$. The same is true if one knows a priori that $\{\rho_{\eps}\}$ is uniformly integrable in $L^2((0,T)\times \R^d)$ (or even $L^2_{\text{loc}}((0,T)\times \R^d)$ if $\omega$ is compactly supported).  \\

\noindent Being not satisfied as we cannot cover the physical dimension $d=3$, our approach is to look for an appropriate assumption on $\omega$ so that we can exploit more the aforementioned nonlocal quantities. We were inspired by the recent work on Landau equation \cite{carrillo2020landau,carrillo2022convergence} where the authors consider the kernel $\omega(x) = e^{-(1+|x|^2)^{1/2}}$ which satisfies the crucial estimate
\begin{equation}\label{eq:estimate_for_Carrillo_kernel} 
|\nabla \omega(x)| \leq C\, \omega(x).
\end{equation}
Such inequalities allow to transform difficult-to-estimate terms by better understood ones, for instance $\varepsilon\, \rho_{\eps} \ast |\nabla \omega_{\eps}|$ can be estimated pointwisely by $\rho_{\eps} \ast \omega_{\eps}$ due to nonnegativity of $\rho_{\eps}$ and $\omega_{\eps}$. Nevertheless, \eqref{eq:estimate_for_Carrillo_kernel} is fairly restrictive - for instance, it excludes compactly supported kernels. In our setting, the following assumption turns out to be successful: 
\begin{Ass}\label{ass:kernel_assumption_double_moll_+_growth}
We assume that $\omega:\R^d \to [0,\infty)$ is a smooth function such that $\int_{\R^d} \omega(x) \diff x = 1$ and $\omega(x) = \omega(-x)$. Moreover, there exists an integrable kernel $f: \R^d \to [0,\infty)$ such that for all $x \in \R^d$
\begin{equation}\label{eq:ass:omega_double_mollification}
(|x| + |x|^2) \, |\nabla \omega(x)| \leq C\, \omega \ast f(x).
\end{equation}
Furthermore, we assume that $\omega$ has sufficient decay at $+\infty$:
\begin{equation}\label{ass:growth_omega}
\lim_{R\to\infty} \sup_{|x|=R} |x|^{d}\, \omega(x) = 0.    
\end{equation}
\end{Ass}
This covers the case of Gaussian $\omega(x) = \frac{1}{(2\pi)^{d/2}} e^{-|x|^2/2}$ and any nonnegative, compactly supported kernel by choosing $f= \omega$, see Lemmas \ref{lem:ass_smooth_compactly_supp} and \ref{lem:exponential_kernels}. Moreover, we can cover the kernel $\omega(x) = e^{-(1+|x|^2)^{1/2}}$ by choosing $f$ more carefully, see Lemma~\ref{lem:omega_like_carrillo}.\\

For the initial condition we suppose that $\rho(t=0)=\rho^{0}$ where $\rho^0$ satisfies 
\begin{equation}\label{Ass:init_cond}
\rho^0 \geq 0,\,\,\,\, \rho^0\in L^1(\R^d) \cap H^1(\R^d),\, \,\,\, \rho^0|\log\rho^0|\in L^1(\R^d), \,\,\,\, |x|^{2}\rho^0\in L^{1}(\R^d).
\end{equation}

Our main result reads: 
\begin{thm}\label{thm:single_eq}
Let $\{\rho_{\eps}\}_{\eps}$ be a sequence of solutions to \eqref{eq:nonlocal_PDE_CH} with initial condition $\rho^0$ satisfying~\eqref{Ass:init_cond}. Then, up to a subsequence not relabeled, $\rho_{\eps} \to \rho$ in $L^p(0,T; L^1(\R^d))$ for all $p \in [1,\infty)$ where $\rho$ is a weak solution of the degenerate Cahn-Hilliard equation~\eqref{eq:local_PDE_CH}.
\end{thm}

Our methods are quite flexible and they allow to study the same question for systems of the type
\begin{equation}\label{eq:PDE_systemIntr}
\partial_{t} \rho^i_{\eps} - \Delta \rho^i_{\eps} - \DIV \left(\rho^i_{\eps} \nabla \Delta \sum_{j=1}^N K^{i,j}_{\eps} \ast \rho^j_{\eps}\right) = 0, \qquad i=1,...,N,
\end{equation}
under some additional structural assumptions. This is discussed in Section~\ref{sect:system} (see Theorem~\ref{thm:thm_for_the_system}).\\

To conclude the introduction, let us mention that similar problems have been studied in the literature for the porous media equation. Up to our knowledgle, the first result of this type was obtained by Lions and Mas-Gallic \cite{MR1821479} for the PDE
$$
\partial_t \rho_{\eps} = \DIV(\rho_{\eps} \nabla \rho_{\eps} \ast \omega_{\eps} \ast \omega_{\eps}) 
$$
Then, the cases of cross-diffusion systems and general nonlinear diffusion equations has been considered in \cite{carrillo2023nonlocal} and \cite{burger2022porous,Hecht2023porous}, respectively. These problems are motivated by the numerical algorithms called particle methods. More precisely, consider $N$ particles moving according to the system of ODEs
$$
X_i'(t) = - \f{1}{N} \sum_{j \neq i} \nabla W(X_i(t) - X_j(t)).
$$
Then, the empirical measure $\mu^N(t) = \frac{1}{N}\sum_{i=1}^N \delta_{X_i(t)}$ solves in the sense of distributions
$$
\partial_t \mu^N - \DIV(\mu^N \,\nabla \mu^N \ast W) = 0
$$
so in the limit $N \to \infty$
$$
\partial_t \mu - \DIV(\mu\, \nabla \mu \ast W) = 0.
$$
If $W \weaks \delta_0$, we recover the porous media equation. For numerical experiments based on this method we refer to \cite{MR1821479}.\\

\noindent Let us also comment that \eqref{eq:nonlocal_PDE_CH} could be called the nonlocal Cahn-Hilliard equation but it should not be confused with the nonlocal effect in the following PDE
\begin{equation}\label{eq:CH_deg_nonlocal}
\partial_t \rho_{\eps} = \DIV \big(\rho_{\eps} \nabla (B_{\eps}[\rho_{\eps}] + F'(\rho_{\eps})) \big), 
\end{equation}
where $F$ is the potential and $B_{\eps}[u]$ is a nonlocal operator approximating $-\Delta u$, i.e.
$$
B_\eps[u](x) = \f{1}{\eps^{2}}(u(x)-\omega_{\eps}\ast u(x))=\f{1}{\eps^{2}}\int_{\Td}\omega_{\eps}(y)(u(x)-u(x-y)) \diff y.
$$
The equation was obtained by Giacomin and Lebowitz \cite{MR1453735} as a derivation of the degenerate Cahn-Hilliard equation 
\begin{equation}\label{eq:CH_deg_local}
\partial_t \rho = \DIV \big(\rho \nabla (-\Delta \rho + F'(\rho)) \big), 
\end{equation}
proposed in \cite{Cahn-Hilliard-1958} to model the dynamics of phase separation in binary mixtures. The question of passing to the limit from \eqref{eq:CH_deg_nonlocal} to \eqref{eq:CH_deg_local} was addressed only recently in \cite{elbar-skrzeczkowski} for a single equation and in \cite{carrillo2023degenerate} for a system. This problem is fairly different from \eqref{eq:nonlocal_PDE_CH} as energy and entropy yields strong compactness of $\{\rho_{\eps}\}$ and $\{\nabla \rho_{\eps}\}$ rather than their mollifications $\{\rho_{\eps} \ast \omega_{\eps}\}$ and $\{\nabla \rho_{\eps} \ast \omega_{\eps}\}$ as in the case of \eqref{eq:nonlocal_PDE_CH}.  We also remark that the same problem was studied in the context of the nondegenerate Cahn-Hilliard equation \cite{MR4408204,MR4093616,MR4198717,MR4248454}
$$
\partial_t \rho_{\varepsilon} = \DIV \nabla \mu_{\eps} ,\qquad 
\mu_{\eps} = B_{\eps}[\rho_{\eps}] + F'(\rho_\eps).
$$
Here, one obtains immediately an estimate on $\{\nabla \mu_{\varepsilon}\}$ (by multiplying by $\mu_{\eps}$) which greatly simplifies identification of the limits. Nevertheless, we point out that in \cite{MR4408204,MR4093616,MR4198717,MR4248454} the difficulty is rather the low regularity of the potential and the kernel.\\

The structure of the paper is as follows. In Section~\ref{sect:examples_kernels} we show that Assumption~\ref{ass:kernel_assumption_double_moll_+_growth} is satisfied by a wide class of kernels. In Section~\ref{sect:uniform_estimates} we  gather the a priori estimates necessary for the proof of the main result, Theorem~\ref{thm:single_eq}, which is proved in Section~\ref{sect:main_result}. In the last Section~\ref{sect:system}, we show how the result can be extended to systems. Finally, Appendix~\ref{app:conv_d=2_higher_integrability} is dedicated to the proof of Theorem~\ref{thm:single_eq} in dimension 2 for a broader class of kernels.

\section{Examples of kernels satisfying Assumption \ref{ass:kernel_assumption_double_moll_+_growth}}\label{sect:examples_kernels}

Three particular classes of kernels are usually found in the literature and we show they satisfy Assumption \ref{ass:kernel_assumption_double_moll_+_growth}. In fact, in all of those examples, we only need to verify condition \eqref{eq:ass:omega_double_mollification}.   
\begin{lem}\label{lem:ass_smooth_compactly_supp}
Let $\omega:\R^d \to [0,\infty)$ be a smooth function such that $\int_{\R^d} \omega(x) \diff x = 1$. Suppose that $\omega$ is supported on the unit ball $\{x \in \R^d: |x| \leq 1\}$ and $\omega>0$ on the interior $\{x \in \R^d: |x| < 1\}$. Then, $\omega$ satisfies \eqref{eq:ass:omega_double_mollification} with $f = \omega$. 
\end{lem}
\begin{proof}
To prove \eqref{eq:ass:omega_double_mollification}, we only need to consider $|x| \leq 1$. By smoothness and compact support of $\omega$, there exists a constant such that $(|x| + |x|^2) \, |\nabla \omega(x)| \leq C$ and it remains to prove that $\inf_{|x|\leq1} \omega \ast \omega(x) > 0$. For any $|x|\leq1$, we see from the formula $\omega\ast\omega(x) = \int_{|y|\leq 1} \omega(x-y) \, \omega(y) \diff y$ that $\omega\ast\omega(x)>0$. As any continuous function attains its infimum on a compact set, the conclusion follows.  
\end{proof}

\begin{lem}\label{lem:exponential_kernels}
Let $\omega(x) = \frac{1}{({2\pi})^{d/2}} e^{-|x|^2/2}$. Then, $\omega$ satisfies \eqref{eq:ass:omega_double_mollification} with $f=\omega$. 
\end{lem}
\begin{proof}
For Gaussians, we know that
\begin{equation*}
\omega \ast \omega(x) 
=  \frac{1}{(2\pi)^d} e^{-|x|^2/4} \int_{\R^d} e^{-\left|\frac{x}{\sqrt{2}} - \sqrt{2} y\right|^2/2} \diff y = \frac{1}{(2\pi)^{d/2}}  e^{-|x|^2/4}.
\end{equation*}
Therefore, since the function $(|x|^2 + |x|^3)\, e^{-|x|^2/4}$ is globally bounded, we find
$$
(|x| + |x|^2)\, |\nabla \omega(x)| \leq C\,(|x|^2 + |x|^3)\, e^{-|x|^2/2} \leq C\, e^{-|x|^2/4} = C\, \omega\ast\omega(x).
$$
\end{proof}

\begin{lem}\label{lem:omega_like_carrillo}
Let $\omega(x) = e^{-(1+|x|^2)^{1/2}}$. Then, $\omega$ satisfies \eqref{eq:ass:omega_double_mollification} with $f(x) = e^{-(1+|x|^2/3)^{1/2}}$.
\end{lem}
\begin{proof}
We need to estimate the convolution $\omega \ast f$ from below so we need to estimate the expression $\sqrt{1+|x-y|^2} + \sqrt{1+|y|^2/3}$ from above. Using $\sqrt{x} + \sqrt{y} \leq \sqrt{2(x+y)}$ and $\sqrt{x+y} \leq \sqrt{x} + \sqrt{y}$, we have
\begin{align*}
    \sqrt{1+|x-y|^2} + \sqrt{1+|y|^2/3} &\leq 
    \sqrt{2} \, \sqrt{2+ |x|^2 - 2x \cdot y + \frac{4}{3} |y|^2}\\
    &\leq \sqrt{2} \sqrt{\frac{1 + |x|^2}{4} + \left|\frac{\sqrt{3}}{2}x - \frac{2}{\sqrt{3}}y \right|^2 + \frac{7}{4}} \\
    &\leq \frac{\sqrt{2}}{2} \sqrt{1+|x|^2} + \sqrt{2}\, \sqrt{\left|\frac{\sqrt{3}}{2}x - \frac{2}{\sqrt{3}}y \right|^2 + \frac{7}{4}}.
\end{align*}
Note that the integral $\int_{\R^d} e^{-\sqrt{2}\, \sqrt{\left|\frac{\sqrt{3}}{2}x - \frac{2}{\sqrt{3}}y \right|^2 + \frac{7}{4}}} \diff y$ is a constant independent of $x$ (by a change of variables). Therefore,
$$
e^{-\frac{\sqrt{2}}{2} \sqrt{1+|x|^2}} \leq C\, \omega \ast f(x). 
$$
We conclude by observing that the function $(|x| + |x|^2)\, e^{-\left(1-\frac{\sqrt{2}}{2}\right) \sqrt{1+|x|^2}}$ is globally bounded
\begin{multline*}
(|x| + |x|^2)\, |\nabla \omega(x)| \leq C\, (|x| + |x|^2)\, e^{- \sqrt{1+|x|^2}} = \\ = 
 C\, (|x| + |x|^2)\, e^{-\left(1-\frac{\sqrt{2}}{2}\right) \sqrt{1+|x|^2}} e^{-\frac{\sqrt{2}}{2} \sqrt{1+|x|^2}} \leq C\,e^{-\frac{\sqrt{2}}{2} \sqrt{1+|x|^2}} \leq C\, \omega \ast f(x).
\end{multline*}
\end{proof}

\section{Uniform estimates for \eqref{eq:nonlocal_PDE_CH} and compactness}\label{sect:uniform_estimates}

The first immediate estimate is the conservation of mass. Integrating the equation in space we obtain an $L^{\infty}(0,T; L^{1}(\R^d))$ control on the solution. Moreover,
the nonlocal equation~\eqref{eq:nonlocal_PDE_CH} comes with an energy/entropy structure. Defining 
\begin{align}\label{eq:energy/entropy}
E_{\eps}[\rho]= \int_{\R^{d}}\f{|\nabla \rho\ast\omega_{\eps}|^{2}}{2} \diff x,\qquad \Phi[\rho]=\int_{\R^{d}}\rho\log(\rho) \diff x,   
\end{align}
we obtain the dissipation equalities:
\begin{align}
&\f{d E_{\eps}[\rho]}{dt}+\int_{\R^d}\left|\Delta \rho\ast\omega_{\eps}\right|^{2} \diff x +\int_{\Omega}\rho\left|\nabla\Delta\rho\ast\omega_{\eps}\ast\omega_{\eps}\right|^{2} \diff x=0, \label{eq:energy_diss}\\
&\f{d\Phi[\rho]}{dt}+\int_{\R^{d}}\f{|\nabla\rho|^{2}}{\rho} \diff x +\int_{\R^{d}}\left|\Delta\rho\ast\omega_{\eps}\right|^{2} \diff x=0. \label{eq:entropy_diss} 
\end{align}
Of course one has to be careful with the entropy equality, as $\rho\log(\rho)$ can be negative when $\rho$ is small and one needs to show that its negative part is integrable. 

\begin{proposition}\label{prop:bounds}
Suppose the initial condition $\rho^0$ satisfies~\eqref{Ass:init_cond}. Then, there exists a unique nonnegative weak solution to \eqref{eq:nonlocal_PDE_CH} satisfying the following bounds, uniformly with respect to~$\eps$:
\begin{enumerate}[label=(\Alph*)]
\item $\{\rho_{\eps}\}_{\eps}\in L^{\infty}(0,T;L^{1}(\R^{d})\cap L\log L(\R^{d}))$,\label{estim3}
\item $\{\p_{t}\rho_{\eps}\}_{\eps}\in L^{2}(0,T;H^{-k}(\R^{d}))$ for some $k$, \label{estim6}
\item $\{\sqrt{\rho_{\eps}}\,\nabla\Delta\rho_{\eps}\ast\omega_{\eps}\ast\omega_{\eps}\}_{\eps}\in L^{2}((0,T)\times\R^{d})$,\label{estim5}
\item $\{\rho_{\eps}\ast\omega_{\eps}\}_{\eps}\in L^{\infty}(0,T;H^{1}(\R^{d}))\cap L^{2}(0,T;H^{2}(\R^{d}))$, \label{estim1}
\item $\{|x|^{2}\rho_{\eps}\}_{\eps}\in L^{\infty}(0,T;L^{1}(\R^{d}))$,\label{estimnew}
\item $\{\nabla\sqrt{\rho_{\eps}}\}_{\eps} \in L^{2}(0,T;L^{2}(\R^{d}))$,\label{estim2}
\item $\{\nabla \rho_{\eps}\}_{\eps}\in L^{2}(0,T; L^1(\R^{d}))$.\label{estim4}
\end{enumerate}
Moreover, we can extract a subsequence such that
\begin{align}
&\rho_{\eps}\to\rho \, \text{ strongly in $L^{p}(0,T;L^{1}(\Rd))$},\, p<\infty, \label{conv1}\\
& \rho_{\eps}\ast \omega_{\eps}\rightharpoonup \rho \, \text{ weakly in $L^{\infty}(0,T;H^{1}(\Rd))\cap L^{2}(0,T;H^{2}(\R^{d}))$},\label{conv2}\\
& \rho_{\eps}\ast \omega_{\eps}\to \rho \,\text{ a.e. and strongly in $L^{2}(0,T;H^{1}_{\text{loc}}(\R^{d}))$}.\label{conv3} 
\end{align}
\end{proposition}

\begin{proof}[Proof of Proposition \ref{prop:bounds}] The existence and uniqueness of solutions is a classical matter as \eqref{eq:nonlocal_PDE_CH} is an advection-diffusion equation with smooth advection (as $\omega_{\eps}$ is smooth).\\
The $L^{1}$ bound in \ref{estim3} is a consequence of mass conservation. \ref{estim5} follows directly from \eqref{eq:energy_diss}. Estimate~\ref{estim6} is a consequence of~\ref{estim3}, \ref{estim5}, Equation~\eqref{eq:nonlocal_PDE_CH} and splitting
$$
\rho_{\eps} \,\nabla\Delta\rho_{\eps}\ast\omega_{\eps}\ast\omega_{\eps} =  \sqrt{\rho_{\eps}} \, \sqrt{\rho_{\eps}}\,\nabla\Delta\rho_{\eps}\ast\omega_{\eps}\ast\omega_{\eps}.
$$
To prove \ref{estim1}, we first deduce from \eqref{eq:energy_diss} bounds on $\nabla \rho_{\eps}\ast\omega_{\eps}$ and $\Delta\rho_{\eps}\ast\omega_{\eps}$. Then, the $L^{\infty}(0,T; L^2(\R^d))$ bound on $\{\rho_{\eps} \ast \omega_{\eps}\}_{\eps}$ follows from bounds on $\{\rho_{\eps}\ast \omega_{\eps} \}_{\eps}$ in $L^{\infty}(0,T; L^1(\R^d))$, $\{\nabla \rho_{\eps} \ast \omega_{\eps}\}_{\eps}$ in $L^{\infty}(0,T; L^2(\R^d))$ and from the Gagliardo-Nirenberg inequality). The estimate in $L^2(0,T; H^2(\R^d))$ is a consequence of the second-order regularizing property of the operator $\Delta$ on the whole space.
\\

To see~\ref{estimnew}, we multiply Equation~\eqref{eq:nonlocal_PDE_CH} by $|x|^{2}$ and obtain after integration by parts 
$$
\f{d}{dt}\int_{\R^d}|x|^{2}\rho_{\eps}\diff x = 2\int_{\R^{d}}\rho_{\eps}\diff x +2 \int_{\R^{d}}\sqrt{\rho_{\eps}} \, x\cdot \sqrt{\rho_{\eps}} \, \nabla\Delta \rho_{\eps}\ast\omega_{\eps}\ast\omega_{\eps}\diff x.
$$
Using \ref{estim3}, \ref{estim5}, the Cauchy-Schwarz inequality  and the Gronwall lemma, we obtain~\ref{estimnew}.\\

It remains to prove the second part of estimate~\ref{estim3} namely the $L\log(L)$ bound on $\rho_{\eps}$. A small difficulty is that the negative part of $\rho_{\eps}\log(\rho_{\eps})$ might not be integrable on the whole space. Nevertheless, as in~\cite{Hecht2023porous}, one can prove that $\rho_{\eps} |\log \rho_{\eps}|_{-}$ is uniformly bounded in $L^{\infty}(0,T; L^1(\R^d))$ by splitting $\R^d$ for $\{x: |\rho_{\eps}|\leq e^{-|x|^2}\}$, $\{x: |\rho_{\eps}| > e^{-|x|^2}\}$ and applying tail estimate \ref{estimnew}. Hence, we can use~\eqref{eq:entropy_diss} and deduce Estimate~\ref{estim2}. Estimate \ref{estim4} follows from~\ref{estim3} and \ref{estim2} by writing $\nabla \rho_{\eps} = \frac{\nabla\rho_{\eps}}{\sqrt{\rho_{\eps}}} \sqrt{\rho_{\eps}}$.\\ 

The convergences \eqref{conv1}--\eqref{conv3} are a consequence of the Lions-Aubin lemma and the Banach-Alaoglu theorem where convergence~\eqref{conv1} has been upgraded from a local to a global one by the tail estimate~\ref{estimnew}. 
\end{proof} 

\section{The proof of the main result}\label{sect:main_result}

We only need to study the term
$
\int_0^T \int_{\R^d} \nabla \varphi \, \rho_{\eps} \, \nabla \Delta (\rho_{\eps} \ast \omega_{\eps} \ast \omega_{\eps}) \diff x \diff t,
$
for test functions $\varphi \in C_c^{\infty}([0,T]\times \R^d)$. Using the properties of the mollifiers
$$
\int_0^T \int_{\R^d} \nabla \varphi \, \rho_{\eps} \, \nabla \Delta (\rho_{\eps} \ast \omega_{\eps} \ast \omega_{\eps}) \diff x \diff t = - 
\int_0^T \int_{\R^d} (\nabla \varphi \, \rho_{\eps}) \ast \nabla \omega_{\eps} \,  \Delta (\rho_{\eps} \ast \omega_{\eps}) \diff x \diff t .
$$
Thanks to the weak convergence \eqref{conv3}, we only need to prove the following strong convergence result
\begin{equation}\label{eq:target}
(\nabla \varphi \, \rho_{\eps}) \ast \nabla \omega_{\eps} \to \nabla \varphi \cdot \nabla \rho + \Delta \varphi\, \rho = \DIV(\nabla \varphi \, \rho) \mbox { strongly in } L^2((0,T)\times\R^d).
\end{equation}
We write $\nabla \varphi(y) = \nabla \varphi(y) - \nabla \varphi(x) + \nabla \varphi(x)$ which results in two terms:
\begin{multline*}
(\nabla \varphi \, \rho_{\eps}) \ast \nabla \omega_{\eps}(x) = \int_{\R^d} \nabla \varphi(y) \, \rho_{\eps} (y) \cdot \, \nabla \omega_{\eps}(x-y) 
\diff y = \nabla \varphi(x) \cdot \nabla  \rho_{\eps} \ast  \omega_{\eps}\, + \\ 
+
\int_{\R^d} (\nabla \varphi(y) - \nabla \varphi(x)) \rho_{\eps}(y) \cdot \nabla \omega_{\eps}(x-y) \diff y  =: I + J.
\end{multline*}
According to \eqref{conv3}, the term $I$ converges strongly in $L^2((0,T)\times\R^d)$ (note that $\varphi$ is compactly supported). The rest of the proof is devoted to the analysis of the term $J$.\\

By Taylor's expansion
$$
\nabla \varphi(y) - \nabla \varphi(x) = D^2 \varphi(x) \,  (y-x)  + R(x,y)
$$
where the term $R$ satisfies $|R| \leq C\, |y-x|^2$. We split $J = J_1 + J_2$, where
$$
J_1 = \int_{\R^d} D^2 \varphi(x) \,  (y-x)  \rho_{\eps}(y) \nabla \omega_{\eps}(x-y) \diff y, \, 
J_2 = \int_{\R^d} R(x,y)\, \rho_{\eps}(y) \nabla \omega_{\eps}(x-y) \diff y. 
$$

\underline{Term $J_1$.} We prove that
$
\lim_{\varepsilon\to0}J_1 = \Delta \varphi\, \rho
$ in $L^2((0,T)\times \R^d)$.
Since $\varphi$ is compactly supported, it is sufficient to prove that
\begin{equation}\label{eq:conv_term_J1}
\int_{\R^d} \rho_{\eps}(y) (x_i-y_i) \partial_j \omega_{\eps}(x-y) \diff y \to -\rho(x)\,\delta_{i,j} \qquad \mbox{ in }  L^2_{\text{loc}}((0,T)\times \R^d).
\end{equation}
Assertion \eqref{eq:conv_term_J1} will be obtained by proving convergence in $L^1((0,T)\times \R^d)$ and uniform boundedness in $L^p((0,T)\times \R^d)$ for some $p>2$. Concerning the convergence in $L^1((0,T)\times\R^d)$, we first change variables
$$
\int_{\R^d} \rho_{\eps}(y) (x_i-y_i) \partial_j \omega_{\eps}(x-y) \diff y = \int_{\R^d} \rho_{\eps}(x-\varepsilon z) z_i\,  \partial_j \omega(z) \diff z
$$
We can estimate in $L^1((0,T)\times\R^d)$ the difference
\begin{multline*}
\int_0^T \int_{\mathbb{R}^d} \left|\int_{\R^d} \left(\rho_{\eps}(x-\varepsilon z) - \rho_{\eps}(x)\right) z_i\,  \partial_j \omega(z) \diff z\right| \diff x \diff t \leq \\
\leq \int_0^T \int_{\mathbb{R}^d} \int_{\R^d} \int_0^1 \varepsilon |\nabla \rho_{\varepsilon}(x-\varepsilon s z)|\,  |z|^2\,  |\nabla \omega(z)| \diff s \diff z \diff x \diff t \leq \varepsilon \|\nabla \rho_{\varepsilon}\|_{L^1_{t,x}} \, \| |z|^2 \nabla \omega \|_{L^1_x},
\end{multline*}
where integrability of $|z|^2 \nabla \omega(z)$ is a consequence of assumption \eqref{eq:ass:omega_double_mollification}:
$$
\int_{\R^d} |z|^2\, |\nabla \omega(z)| \diff z \leq C\,\int_{\R^d} \omega\ast f(z) \diff z = C\,\|\omega\|_{L^1}\,\|f\|_{L^1} \leq C.
$$
Therefore, it is sufficient to study the term $\rho_{\eps}(x) \int_{\R^d} z_i \partial_j \omega(z) \diff z$ which equals $-\rho_{\eps}(x)$ because
$$
\int_{\R^d} z_i \partial_j \omega(z) \diff z = - \delta_{i,j} \int_{\R^d} \omega(z) \diff z = -\delta_{i,j}, 
$$
where the boundary term vanishes thanks to \eqref{ass:growth_omega}. The conclusion follows because $\rho_{\eps}$ is strongly convergent in $L^1((0,T)\times\R^d)$, cf. \eqref{conv1}.\\

Concerning the uniform boundedness in $L^p((0,T)\times\R^d)$ with $p>2$, by nonnegativity of $\rho_{\eps}$, definition of $\omega_{\eps}(x) = \frac{1}{\varepsilon^d}\omega\left(\frac{x}{\varepsilon}\right)$ and assumption \eqref{eq:ass:omega_double_mollification}, 
\begin{equation}\label{eq:estimate_by_double_convolution}
\begin{split}
\left|\int_{\R^d} \rho_{\eps}(y) (x_i-y_i) \partial_j \omega_{\eps}(x-y) \diff y \right| = \left|\int_{\R^d} \rho_{\eps}(y) \frac{(x_i-y_i)}{\varepsilon^{d+1}} \, \partial_j\omega\left(\frac{x-y}{\varepsilon} \right)  \diff y \right|\leq \\
\leq C\,\int_{\R^d} \rho_{\eps}(y)\, \frac{1}{\varepsilon^d} \omega\ast f\left(\frac{x-y}{\varepsilon} \right) \diff y. 
\end{split}
\end{equation}
A change of variables shows that 
\begin{equation}\label{eq:change_variables_double_mollifiers}
\begin{split}
\frac{1}{\varepsilon^d} \omega\ast f\left(\frac{x-y}{\varepsilon} \right) &= \frac{1}{\varepsilon^d}  \int_{\R^d} \omega(z) \, f\left(\frac{x-y}{\varepsilon} - z \right) \diff z =\\ &= \frac{1}{\varepsilon^{2d}}
\int_{\R^d} \omega\left(\frac{z}{\varepsilon}\right) \, f\left(\frac{x-y-z}{\varepsilon} \right) \diff z = \omega_{\eps}\ast f_{\eps}(x-y),
\end{split}
\end{equation}
where $f_{\eps}(x) := \frac{1}{\eps^d} f\left(\frac{x}{\varepsilon} \right)$. Due to \eqref{eq:estimate_by_double_convolution}, 
$$
\left|\int_{\R^d} \rho_{\eps}(y) (x_i-y_i) \partial_j \omega_{\eps}(x-y) \diff y \right| \leq C\, \rho_{\eps} \ast \omega_{\eps} \ast f_{\eps}(x).
$$
Note that by the Gagliardo-Nirenberg inequality and uniform bound in $L^{\infty}(0,T; H^1(\R^d))$, $\{\rho_{\eps} \ast \omega_{\eps}\}_{\eps}$ is bounded in $L^{\infty}(0,T; L^{\frac{2\,d}{d-2}}(\R^d))$ where $\frac{2\,d}{d-2} > 2$. The same is true for $\{\rho_{\eps} \ast \omega_{\eps} \ast f_{\eps}\}_{\eps}$ by the Young convolutional inequality. The conclusion follows.  \\

\underline{Term $J_2$.} We prove that $\lim_{\varepsilon \to 0} J_2 = 0$. By $|R| \leq C\, |y-x|^2$, it is sufficient to prove
$$
\int_{\R^d} \rho_{\eps}(y) \, |x-y|^2 \, |\nabla \omega_{\eps}(x-y)| \diff y \to 0 \qquad \mbox{ in } L^2((0,T)\times\R^d).
$$
Again, we want to use assumption \eqref{eq:ass:omega_double_mollification}. By definition of $\omega_{\eps}$:
\begin{multline*}
\int_{\R^d} \rho_{\eps}(y) \, |x-y|^2 \, |\nabla \omega_{\eps}(x-y)| \diff y = 
\varepsilon \int_{\R^d} \rho_{\eps}(y) \, \frac{1}{\varepsilon^d}\, \left|\frac{x-y}{\varepsilon}\right|^2 \, \left|\nabla \omega\left(\frac{x-y}{\varepsilon}\right)\right| \diff y \leq \\
\leq \varepsilon \int_{\R^d} \rho_{\eps}(y) \, \frac{1}{\varepsilon^d} \omega\ast f\left(\frac{x-y}{\varepsilon}\right) = \varepsilon\,  \rho_{\varepsilon} \ast \omega_{\eps} \ast f_{\eps}(x),
\end{multline*}
where in the last line we applied \eqref{eq:change_variables_double_mollifiers}. By the Young convolutional inequality, $\rho_{\eps} \ast \omega_{\eps} \ast f_{\eps}$ on the (RHS) is bounded in $L^2((0,T)\times\R^d)$ so the conclusion follows. 

\section{Extension to systems}
\label{sect:system}

Motivated by \cite{Hecht2023porous}, we consider the system of $N$ equations
\begin{equation}\label{eq:PDE_system}
\partial_{t} \rho^i_{\eps} - \Delta \rho^i_{\eps} + \DIV \left(\rho^i_{\eps} \nabla \Delta \sum_{j=1}^N K^{i,j}_{\eps} \ast \rho^j_{\eps}\right) = 0,
\end{equation}
where $1\leq i \leq N$ and the kernels $K^{i,j}_{\eps}$ are of the form 
$$
K^{i,j}_\eps = \sum_{k=1}^N \alpha^{i,k}\,\alpha^{j,k} \, \omega^i_{\eps} \ast \omega^j_{\eps}, 
$$
where the $\omega^i$ are kernels satisfying Assumption \ref{ass:kernel_assumption_double_moll_+_growth}. The coefficients $\{\alpha^{i,k}\}$ form a matrix~$A$ and we assume it is invertible. Under these assumptions, for any set of functions $\eta_1$, ..., $\eta_N$:
$$
\sum_{i,j=1}^N \int_{\R^d} \eta^{i}\,K^{i,j}_\eps \ast \eta^{j} = \int_{\R^d} \sum_{k=1}^N \left( \sum_{i=1}^N \alpha^{i,k}\, \eta^{i} \ast \omega^i_{\eps} \right)^2
$$
so that
\begin{equation}\label{eq:lower_bound_due_to_aij}
\widetilde{C}\,\|A^{-1}\|^2 \int_{\R^d}\sum_{i=1}^N \left( \eta^{i} \ast \omega^i_{\eps} \right)^2 \geq \sum_{i,j=1}^N \int_{\R^d} \eta^{i}\,K^{i,j}_\eps \ast \eta^{j}  \geq C\, \|A^{-1}\|^2 \int_{\R^d}\sum_{i=1}^N \left( \eta^{i} \ast \omega^i_{\eps} \right)^2.
\end{equation}

\begin{thm}\label{thm:thm_for_the_system}
Let $\{\rho^i_{\eps}\}_{\eps}$ be a sequence of solutions to \eqref{eq:PDE_system} with initial condition $\rho^{0,i}$ satisfying \eqref{Ass:init_cond}. Then, for $i=1,...,N$, and for a subsequence not relabeled, $\rho^i_{\eps} \to \rho^i$ in $L^p(0,T; L^1(\R^d))$ for all $p \in [1,\infty)$ where $\rho^i$ is a weak solution of 
$$
\partial_{t} \rho^i - \Delta \rho^i + \DIV \left(\rho^i\,  \nabla \Delta \sum_{j=1}^N K^{i,j} \, \rho^j\right) = 0, \qquad K^{i,j} = \sum_{k=1}^N \alpha^{i,k}\,\alpha^{j,k}.
$$
\end{thm}

We first extend the uniform bounds in Proposition \ref{prop:bounds} to the case of system \eqref{eq:PDE_system}.

\begin{proposition}\label{prop:bounds_systems}
Suppose that for all $i =1, ..., N$, the initial conditions $\rho^{0,i}$ satisfy \eqref{Ass:init_cond}. Then, the nonnegative solution to \eqref{eq:PDE_system} satisfies the following bounds, uniformly with respect to~$\eps$:
\begin{enumerate}[label=(\Alph*)]
\item $\{\rho^i_{\eps}\}_{\eps}\in L^{\infty}(0,T;L^{1}(\R^{d})\cap L\log L(\R^{d}))$,\label{estim3_sys}
\item $\{\p_{t}\rho^i_{\eps}\}_{\eps}\in L^{2}(0,T;H^{-k}(\R^{d}))$ for some $k$, \label{estim6_sys}
\item $\{\sqrt{\rho^i_{\eps}}\, \nabla \Delta \sum_{j=1}^N K^{i,j}_{\eps} \ast \rho^j_{\eps} \}_{\eps}\in L^{2}((0,T)\times\R^{d})$,\label{estim5_sys}
\item $\{\rho^i_{\eps}\ast\omega^i_{\eps}\}_{\eps}\in L^{\infty}(0,T;H^{1}(\R^{d}))\cap L^{2}(0,T;H^{2}(\R^{d}))$, \label{estim1_sys}
\item $\{|x|^{2}\rho^i_{\eps}\}_{\eps}\in L^{\infty}(0,T;L^{1}(\R^{d}))$,\label{estimnew_sys}
\item $\{\nabla\sqrt{\rho^i_{\eps}}\}_{\eps} \in L^{2}(0,T;L^{2}(\R^{d}))$,\label{estim2_sys}
\item $\{\nabla \rho^i_{\eps}\}_{\eps}\in L^{2}(0,T; L^1(\R^{d}))$,\label{estim4_sys}

\end{enumerate}
Moreover, we can extract a subsequence such that for all $i = 1, ..., N$
\begin{align}
&\rho^i_{\eps}\to\rho^i \, \text{ strongly in $L^{p}(0,T;L^{1}(\Rd))$},\, p<\infty \label{conv1_sys}\\
& \rho^i_{\eps}\ast \omega^i_{\eps}\rightharpoonup \rho^i \, \text{ weakly in $L^{\infty}(0,T;H^{1}(\Rd))\cap L^{2}(0,T;H^{2}(\R^{d}))$},\label{conv2_sys}\\
& \rho^i_{\eps}\ast \omega^i_{\eps}\to \rho^i \,\text{ a.e. and strongly in $L^{2}(0,T;H^{1}_{\text{loc}}(\R^{d}))$}.\label{conv3_sys} 
\end{align}
\end{proposition}
\begin{proof}
The proof is almost the same as the proof of Proposition \ref{prop:bounds}. The only difficulty is to obtain the energy and entropy identities corresponding to \eqref{eq:energy_diss} and \eqref{eq:entropy_diss}, respectively.\\

Concerning the energy, we multiply \eqref{eq:PDE_system} with $\Delta \sum_{j=1}^N K^{i,j} \ast \rho^j_{\eps}$, integrating in space and summing up for $i=1,...,N$ yields
\begin{multline*}
\frac{C\,\|A^{-1}\|^2}{2} \sum_{i=1}^N \int_{\R^d} |\nabla \rho^{i}_{\eps} \ast  \omega^i_{\eps} (t,x)|^2 \diff x + C\, \|A^{-1}\|^2 \, \sum_{i=1}^N \int_0^t \int_{\R^d} |\Delta \rho^{i}_{\eps} \ast \omega^{i}|^2 \diff x \diff s \, + \\
+ \sum_{i=1}^ N \int_0^t \int_{\R^d}\rho^i_{\eps} \left| \nabla \Delta \sum_{j=1}^N K^{i,j}_{\eps} \ast \rho^j_{\eps} \right|^2 \diff x \diff s \leq \frac{\widetilde{C}\,\|A\|^2}{2} \sum_{i=1}^N \int_{\R^d} |\nabla \rho^{i,0}|^2 \diff x.
\end{multline*}
This identity implies \ref{estim5_sys} and \ref{estim1_sys}. Estimate \ref{estim6_sys} follows from the PDE \eqref{eq:PDE_system} and \ref{estim5_sys}. Thanks to \ref{estim5_sys}, we also deduce \ref{estimnew_sys}.\\

Concerning the entropy, we multiply \eqref{eq:PDE_system} with $\log \rho^{i}_{\eps}$, integrate in space and sum up to obtain
$$
\sum_{i=1}^{N} \partial_t \int_{\R^d} \rho_{\eps}^i\, (\log \rho_{\eps}^i - 1) \diff x  + \sum_{i=1}^{N} \int_{\R^d} \frac{|\nabla \rho_{\eps}^i|^2}{\rho_{\eps}^i} \diff x
+ \sum_{i,j=1}^N \int_{\R^d} K^{i,j}\ast \Delta \rho^{j}_{\eps} \, \Delta \rho^{i}_{\eps} \diff x = 0.
$$
Applying \eqref{eq:lower_bound_due_to_aij} with $\eta_i = \Delta \rho^{i}_{\eps}$, we deduce
$$
\sum_{i=1}^{N} \partial_t \int_{\R^d} \rho_{\eps}^i\, (\log \rho_{\eps}^i - 1) \diff x  + \sum_{i=1}^{N} \int_{\R^d} \frac{|\nabla \rho_{\eps}^i|^2}{\rho_{\eps}^i} \diff x
+ C\, \|A^{-1}\|\, \sum_{i=1}^N \int_{\R^d} |\Delta \rho^{i}_{\eps} \ast \omega^{i}_{\eps}|^2 \diff x \leq 0.
$$
As in the proof of Proposition \ref{prop:bounds}, one may check that $\rho_{\eps}^i\, |\log \rho_{\eps}^i|_{-}$ is uniformly bounded in $L^{\infty}(0,T; L^1(\R^d))$ which implies \ref{estim3_sys}, \ref{estim2_sys} and~\ref{estim4_sys}.
 The convergences \eqref{conv1_sys}--\eqref{conv3_sys} easily follow from the estimates.
\end{proof}
\begin{proof}[Proof of Theorem \ref{thm:thm_for_the_system}]
By linearity, we only need to explain how to pass to the limit in the term 
$$
\int_0^T \int_{\R^d} \nabla \varphi \, \rho^i_{\eps} \, \nabla \Delta (\rho^j_{\eps} \ast \omega^j_{\eps} \ast \omega^i_{\eps}) \diff x \diff t = -\int_0^T \int_{\R^d} (\nabla \varphi \, \rho^i_{\eps}) \ast \nabla \omega^{i}_{\eps} \, \Delta (\rho^j_{\eps} \ast \omega^j_{\eps}) \diff x \diff t .
$$
However, in the proof of Theorem \ref{thm:single_eq}, we proved that 
$$
(\nabla \varphi \, \rho^i_{\eps}) \ast \nabla \omega^{i}_{\eps} \to \nabla \varphi \cdot \nabla \rho^i +\Delta \varphi\, \rho^i \mbox { strongly in } L^2((0,T)\times\R^d),
$$
see \eqref{eq:target}. Thanks to the weak convergence $\Delta (\rho^j_{\eps} \ast \omega^j_{\eps}) \weak \Delta \rho^j_{\eps}$ in \eqref{conv3_sys}, we conclude the proof. 
\end{proof}

\appendix
\section{Proof of the convergence for general kernels and $d=2$ }\label{app:conv_d=2_higher_integrability}

In dimension $d=2$ another proof of the main result uses weaker assumptions, namely
$$
d=2, \qquad  y \, \omega(y) \in L^1(\R^d), \qquad  y\, \nabla \omega(y) \in L^2(\R^{d}).
$$

As in the main proof, we only need to study term $
\int_0^T \int_{\R^d} \nabla \varphi \, \rho_{\eps} \, \nabla \Delta (\rho_{\eps} \ast \omega_{\eps} \ast \omega_{\eps}) \diff x \diff t,
$
where $\varphi \in C_c^{\infty}([0,T]\times \R^d)$. Integrating by parts
$$
\int_0^T \int_{\R^d} \nabla \varphi \, \rho_{\eps} \, \nabla \Delta (\rho_{\eps} \ast \omega_{\eps} \ast \omega_{\eps}) \diff x \diff t = - \int_0^T \int_{\R^d} \DIV(\nabla \varphi \rho_{\eps})\ast \omega_{\eps}\,  \Delta (\rho_{\eps} \ast \omega_{\eps} ) \diff x \diff t $$
According to the a priori estimate \ref{estim1}, we need to prove that
$
\DIV(\nabla \varphi \rho_{\eps})\ast \omega_{\eps}  
$
converges strongly in $L^2((0,T)\times\R^d)$. We introduce the truncation operator $$
T_M(\rho) = \begin{cases}
\rho &\mbox{ if } \rho \leq M,\\
M &\mbox{ if } \rho >M,
\end{cases}
$$
so that splitting $\rho_{\eps} = \rho_{\eps} - T_M(\rho_{\eps}) + T_M(\rho_{\eps})$ we have
\begin{align*}
\DIV(\nabla \varphi \rho_{\eps})\ast \omega_{\eps} &= \DIV(\nabla \varphi T_M(\rho_{\eps}))\ast \omega_{\eps} + \DIV(\nabla \varphi(\rho_{\eps} - T_M(\rho_{\eps})))\ast \omega_{\eps} \\
&= (\Delta \varphi T_M(\rho_{\eps}))\ast \omega_{\eps} + ( \nabla \varphi \cdot \nabla \rho_{\eps} \mathds{1}_{\rho_{\eps}\leq M})\ast \omega_{\eps} + (\nabla \varphi (\rho_{\eps} - T_M(\rho_{\eps})))\ast \nabla \omega_{\eps} \\
&=: I_1 + I_2 + I_3.
\end{align*}
The parameter $M$ will be chosen later in terms of $\varepsilon$ so that $M \to \infty$ as $\varepsilon \to 0$.
\underline{Term $I_1$.} We write
\begin{align*}
I_1(t,x) &= \int_{\R^d} (\Delta \varphi(x-y) - \Delta \varphi(x)) T_M(\rho_{\eps})(x-y)\, \omega_{\eps}(y) \diff y  + \Delta \varphi (T_M(\rho_{\eps}))\ast \omega_{\eps}\\
&= I_1^A + I_1^B.
\end{align*}
As $|\Delta \varphi(x-y) - \Delta \varphi(x)| \leq C\, |y|$, we can estimate
$$
\|I_1^A\|_{L^2_{t,x}} \leq \sqrt{M}\, \|\sqrt{\rho_{\eps}}\|_{L^2_{t,x}} \, \||y| \omega_{\eps}(y)\|_{L^1} \leq \varepsilon\, \sqrt{M} \, \|\sqrt{\rho_{\eps}}\|_{L^2_{t,x}}\, \||y| \omega(y)\|_{L^1} 
$$
so that $\|I_1^A\|_{L^2_{t,x}} \leq C\, \varepsilon\, \sqrt{M}$. Furthermore, note that the term $I_1^B$ is compact in $L^2((0,T)\times\R^d)$ whenever $M \to \infty$, $\varepsilon \to 0$. To see this, first note that it is sufficient to establish local compactness as $\varphi$ is compactly supported. The latter can be proved by the Vitali theorem: we have convergence in measure (even in $L^1_{\text{loc}}((0,T)\times\R^d)$) of $T_M(\rho_{\eps})\ast\omega_{\eps}$ and uniform integrability thanks to the pointwise estimate
$$
0 \leq T_M(\rho_{\eps})\ast\omega_{\eps} \leq \rho_{\eps} \ast \omega_{\eps}
$$
since $\rho_{\eps} \ast \omega_{\eps}$ is compact in $L^2_{\text{loc}}((0,T)\times\R^d)$. We conclude that
\begin{equation}\label{eq:compactness_I_1^B}
I_1^B \to \Delta \varphi \, \rho \, \mbox{ in } L^2((0,T)\times \R^d)  \qquad \mbox{ when } \eps \to 0, M \to \infty.
\end{equation}
\underline{Term $I_2$.} We have
\begin{align*}
I_2(t,x) =& \int_{\R^d} (\nabla \varphi(x-y) - \nabla \varphi(x)) \, \nabla \rho_{\eps}(x-y) \mathds{1}_{\rho_{\eps}(x-y) \leq M} \, \omega_{\eps}(y) \diff y \\&+  \nabla \varphi \cdot (\nabla \rho_{\eps} \mathds{1}_{\rho_{\eps}\leq M})\ast \omega_{\eps}=:I_2^A + I_2^B.
\end{align*}
As $(\nabla \varphi(x-y) - \nabla \varphi(x)) \leq C\,|y|$ and $|\nabla \rho_{\eps}(x-y)| \mathds{1}_{\rho_{\eps}(x-y) \leq M} \leq \sqrt{M} \frac{|\nabla \rho_{\eps}(x-y)|}{\sqrt{\rho_{\eps}(x-y)}}$, we can estimate the term $I_2^A$ as follows
$$
\|I_2^A\|_{L^2_{t,x}} \leq \sqrt{M} \, \left\| \frac{\nabla \rho_{\eps}}{
\sqrt{\rho_{\eps}}} \right\|_{L^2_{t,x}} 
\||y| \omega_{\eps}(y)\|_{L^1} \leq
\varepsilon\,\sqrt{M}\, \left\| \frac{\nabla \rho_{\eps}}{
\sqrt{\rho_{\eps}}} \right\|_{L^2_{t,x}} 
\||y| \omega(y)\|_{L^1}
$$
so that $|I_2^A| \leq C\,\varepsilon\,\sqrt{M}$ according to estimate \ref{estim2}.\\

\underline{Term $I_3$.} We write 
\begin{align*}
    I_3(t,x) =& \int_{\R^d} (\nabla \varphi(x-y) - \nabla \varphi(x))\, (\rho_{\eps}(x-y) - T_M(\rho_{\eps}(x-y)))  \nabla \omega_{\eps}(y) \diff y \\
    &+ \nabla \varphi \, (\rho_{\varepsilon} - T_M(\rho_{\varepsilon})) \ast \nabla \omega_{\varepsilon} = I_3^A + I_3^B.
\end{align*}
We observe that $|\nabla \varphi(x-y) - \nabla \varphi(x)|\leq C\,|y|$ and $|\rho_{\eps} - T_M(\rho_{\eps})|\leq 2\,\rho_{\eps}\,\mathds{1}_{\rho_{\eps}\geq M}$ so the term $I_3^A$ can be estimated as
$$
\|I_3^A\|_{L^2_{t,x}} \leq C\,  \| ( \rho_{\eps} \mathds{1}_{\rho_{\eps}\geq M}) \ast (|y||\nabla \omega_{\eps}(y)|)\|_{L^2_{t,x}}.  
$$
By the Gagliardo-Nirenberg-Sobolev inequality, we get that $\{\rho_{\eps}\}$ is uniformly bounded in $L^2((0,T)\times\R^d)$. Therefore,  
$$
\|I_3^A\|_{L^2_{t,x}} \leq C\, \| \rho_{\eps} \mathds{1}_{\rho_{\eps}\geq M}\|_{L^2_{t} L^1_x} \, \||y|\nabla \omega_{\eps}(y)) \|_{L^2} \leq 
C\, \|\rho_{\eps}\|_{L^2_{t,x}} \, \| \mathds{1}_{\rho_{\eps}\geq M} \|_{L^{\infty}_t L^2_x} \, \||y|\nabla \omega_{\eps}(y)) \|_{L^2}.
$$
It remains to estimate $\| \mathds{1}_{\rho_{\eps}\geq M} \|_{L^{\infty}_t L^2_x}$ and $\, \||y|\nabla \omega_{\eps}(y)) \|_{L^2}$. We have
$$
\| \mathds{1}_{\rho_{\eps}\geq M} \|_{L^{\infty}_t L^2_x} \leq \sup_{t\in(0,T)} \left(\int_{\R^d} \frac{\rho_{\eps} \log\rho_{\eps}}{M \log(M)} \diff x \right)^{1/2}  \leq \frac{C}{M^{1/2} \log^{1/2}M},
$$
$$
\||y|\nabla \omega_{\eps}(y)) \|^2_{L^2} \leq \int_{\R^d} \frac{1}{\eps^{2d+2}} |y|^2 \, \left|\nabla \omega\left(\frac{y}{\eps}\right)\right|^2 \diff y = \frac{1}{\eps^{d}}\int_{\R^d} |y|^2 \left|\nabla \omega\left(y\right)\right|^2 \diff y \leq \frac{C}{\varepsilon^2},
$$
since $d=2$ and using mass conservation \ref{estim3}. We conclude that
$$
\|I_3^A\|_{L^2_{t,x}} \leq \frac{C}{\varepsilon\, M^{1/2} \log^{1/2}M}.
$$

\underline{The conclusion.} Note that the terms $I_2^B$ and $I_3^B$ combine to
$$
I_2^B + I_3^B = \nabla \varphi \nabla (\rho_{\eps} \ast \omega_{\eps})
$$
which is compact in $L^2((0,T)\times\R^d)$ and converges to $\nabla \varphi \, \nabla \rho$. Therefore, 
\begin{equation}\label{eq:decomposition_with_error_controlled}
\DIV(\nabla \varphi \rho_{\eps})\ast \omega_{\eps} = \Delta \varphi (T_M(\rho_{\eps}))\ast \omega_{\eps} +  \nabla \varphi \nabla (\rho_{\eps} \ast \omega_{\eps}) + R,
\end{equation}
where the first two terms are compact in $L^2((0,T)\times\R^d)$ (see also \eqref{eq:compactness_I_1^B}) while
$$
\|R\|_{L^2_{t,x}} \leq C\,\varepsilon\,\sqrt{M} + \frac{C}{\varepsilon\,\sqrt{M}\, \log^{1/2}(M)}.
$$
The conclusion follows by choosing $M$ such that $\varepsilon^2 M \log^{1/2} M = 1$.

\begin{rem}
In arbitrary dimension $d$, if we knew that $\{\rho_{\eps}\}_{\eps}$ is uniformly integrable in $L^2((0,T)\times\R^d)$, i.e.
\begin{equation}\label{eq:rho_uniform_integrability}
\lim_{\varepsilon \to 0}  \|\rho_{\eps} \mathds{1}_{\rho_{\eps} > \frac{1}{\eps}}\|_{L^2_{t,x}} = 0,
\end{equation}
we could conclude in an easier way. Indeed, assuming that $y \, \nabla \omega(y) \in L^1(\R^d)$, one can estimate
$$
\|I_3^A\|_{L^2_{t,x}} \leq \| \psi \rho_{\eps} \mathds{1}_{\rho_{\eps}\geq M} \|_{L^2_{t,x}}\, \| |y|\nabla \omega_{\eps}(y)\|_{L^1} \leq C\, \| \psi \rho_{\eps} \mathds{1}_{\rho_{\eps}\geq M} \|_{L^2_{t,x}}.
$$
Choosing $M = \frac{1}{\eps}$, we conclude. The condition \eqref{eq:rho_uniform_integrability} can be relaxed to be satisfied locally when $\omega$ is compactly supported. We stress that we do not have any a priori estimate implying \eqref{eq:rho_uniform_integrability}. 
\end{rem}

\bibliographystyle{abbrv}
\bibliography{fastlimit}
\end{document}